\date{}
\documentclass[12pt, reqno]{amsart}
\usepackage{latexsym, amsfonts, amsmath, amssymb, amsthm}
\usepackage[dvips]{graphicx}
\allowdisplaybreaks

\usepackage{setspace}
\usepackage[left=2.9cm,top=2.9cm,right=2.9cm,bottom=2.9cm,foot=1cm,head=1cm]{geometry}

\begin{document}
\newtheorem{thrm}{Theorem}[section]
\newtheorem{cor}[thrm]{Corollary}
\newtheorem{lem}[thrm]{Lemma}
\newtheorem{mainthrm}{Theorem}
\def\theequation{\arabic{section}.\arabic{equation}}
\setcounter{equation}{0}

\setlength{\parindent}{0in} %Don't have to use \noindent :)  YAY!!
\renewcommand{\baselinestretch}{1.4}
\newcommand\ds{\displaystyle}

\title{Asymptotic zero distribution of a class of hypergeometric polynomials}
\author[K. A.~Driver]{K. A.~Driver} 
\address{K. A.~Driver\\ Department of Mathematics and Applied Mathematics\\ University of Cape Town, Private Bag X3, Rondebosch 7700, South Africa} 
\email{kathy.driver@uct.ac.za} 

\author[S. J.~Johnston]{S. J.~Johnston} 
\address{S. J.~Johnston\\ School of Mathematics\\ University of the Witwatersrand, Private Bag 3, Wits 2050, Johannesburg, South Africa} 
\email{sarahjane.johnston@wits.ac.za} 

\thanks{Research by the first author is supported by the National Research Foundation under grant number 2053730.}

\begin{abstract}
We prove that the zeros of ${}_2F_1\left(-n,\frac{n+1}{2};\frac{n+3}{2};z\right)$ asymptotically approach the section of the lemniscate $\left\{z: \left|z(1-z)^2\right|=\frac{4}{27}; \textrm{Re}(z)>\frac{1}{3}\right\}$ as $n\rightarrow \infty$.  In recent papers (cf. \cite{KMF}, \cite{orive}), Mart\'inez-Finkelshtein and Kuijlaars and their co-authors have used Riemann-Hilbert methods to derive the asymptotic zero distribution of Jacobi polynomials $P_n^{(\alpha_n,\beta_n)}$ when the limits $\ds A=\lim_{n\rightarrow \infty}\frac{\alpha_n}{n}$ and $\ds B=\lim_{n\rightarrow \infty}\frac{\beta_n}{n}$ exist and lie in the interior of certain specified regions in the $AB$-plane.  Our result corresponds to one of the transitional or boundary cases for Jacobi polynomials in the Kuijlaars Mart\'inez-Finkelshtein classification.
\end{abstract}

\maketitle
\pagenumbering{arabic}
\setcounter{page}{1}

{\small
\emph{Mathematics Subject Classication: 33C05, 30C15.}}

\section{Introduction}

The general hypergeometric function is defined by

\[{}_pF_q(a_1, \ldots ,a_p;b_1, \ldots ,b_q;z)=\sum_{k=0}^{\infty} \frac{(a_1)_k \ldots (a_p)_k}{(b_1)_k \ldots (b_q)_k}\; \frac{z^k}{k!}\quad , \quad |z|<1 \]

where
\[(\alpha)_k=
\begin{cases}
\alpha(\alpha+1) \ldots (\alpha+k-1)& \quad , \quad k\geq 1, \medspace\\
1& \quad , \quad k=0,\; \alpha\neq 0
\end{cases} \]

\bigskip
is Pochhammer's symbol.  When one of the numerator parameters is equal to a negative integer, say $a_1 = -n, \; n\in \mathbb{N}$, the series terminates and the function reduces to a polynomial of degree $n$ in $z$.\\

The Jacobi polynomial $P_n^{(\alpha,\beta)}(x)$ can be defined in terms of a ${}_2F_1$ hypergeometric polynomial (cf. \cite{rain}, p. 254), viz, $$P_n^{(\alpha,\beta)}(z)=\frac{(1+\alpha)_n}{n!}\; {}_2F_1 \left(-n,1+\alpha +\beta +n; 1+\alpha; \frac{1-z}{2}\right).$$ The study of the zeros of Jacobi polynomials therefore gives direct information about the zeros of the corresponding ${}_2F_1$ hypergeometric polynomials and vice versa.\\

In \cite{orive}, Mart\'inez-Finkelshtein, Mart\'inez-Gonz\'alez and Orive consider the asymptotic zero distribution of the Jacobi polynomials $P_n^{(\alpha_n,\beta_n)}$ where the limits 

\begin{equation}\label{limAB}
A=\lim_{n\rightarrow \infty} \frac{\alpha_n}{n} \quad \textrm{ and } \quad B =\lim_{n\rightarrow \infty} \frac{\beta_n}{n}
\end{equation}

exist.  They distinguish five cases depending on the values of $A$ and $B$ and prove results for the ``general'' cases where $A$ and $B$ lie in the interior of certain regions in the plane.\\

The boundary lines are ``non-general'' cases, some of which have been studied (cf. \cite{boggsduren}, \cite{borweinchen}, \cite{moller}, \cite{driverduren}, \cite{durenguillou}), always in the context of the corresponding hypergeometric function.\\

We find the asymptotic zero distribution of ${}_2F_1\left(-n,\frac{n+1}{2};\frac{n+3}{2}; z \right)$ as $n\to\infty$, which corresponds to a non-general case not previously studied. Our results, taken in conjunction with those in \cite{boggsduren}, \cite{moller}, \cite{driverduren} and \cite{durenguillou} suggest that there may be a general result for the asymptotic zero distribution of the class of Jacobi polynomials $P_n^{(\alpha_n,\beta_n)}$ where the limits in (\ref{limAB}) are $$\lim_{n\rightarrow \infty}\frac{\alpha_n}{n}=k  \textrm{ and } \lim_{n\rightarrow \infty}\frac{\beta_n}{n}= -1.$$

\bigskip
We shall prove the following theorem.

\begin{mainthrm}
The zeros of the hypergeometric polynomial $${}_2F_1\left(-n, \frac{n+1}{2}; \frac{n+3}{2};z \right)$$ approach the section of the lemniscate $$\left\{z: \left|z(1-z)^2\right|=\frac{4}{27}; \textrm{Re}(z)>\frac{1}{3}\right\},$$ as $n\rightarrow \infty$.
\end{mainthrm}

Our method, which follows the same approach as that used in \cite{durenguillou}, involves the asymptotic analysis of an integral of the form 

\begin{equation}\label{genint}
A_n\; \int_{0}^{1} \left[f_z(t)\right]^n\; dt
\end{equation}

where $A_n$ is a constant involving $n$ and $f_z(t)$ is a polynomial in the complex variable $t$ and analytic in $z$.\\

Expressing ${}_2F_1\left(-n,\frac{n+1}{2};\frac{n+3}{2};z\right)$ in the form given in (\ref{genint}) can be done either by substituting $t^2$ for $t$ in the Euler integral formula for ${}_2F_1$ functions (cf. \cite{rain}, p. 47), given by $${}_2F_1(a,b;c;z)=\frac{\Gamma(c)}{\Gamma(b)\Gamma(c-b)}\; \int_{0}^{1}t^{b-1}(1-t)^{c-b-1}(1-zt)^{-a}\;dt,$$ or, more directly, using an integral representation for ${}_3F_2$ hypergeometric functions, namely, (cf. \cite{johnston2}, Cor 2.2) \[{}_3F_2\left(-n,\frac{b}{2},\frac{b+1}{2};\frac{c}{2},\frac{c+1}{2};z\right)=\frac{\Gamma (c)}{\Gamma (b) \Gamma (c-b)}\int^{1}_{0} t^{b-1} (1-t)^{c-b-1} (1-zt^2)^n dt,\] for  $\text{Re}(c)>\text{Re}(b)>0$.

\bigskip
Putting $b=n+1$ and $c=n+2$, we obtain 

\begin{equation}\label{smiley}
{}_2F_1\left(-n,\frac{n+1}{2};\frac{n+3}{2};z\right) = (n+1) \int^{1}_{0} \left[t(1-zt^2)\right]^n dt= (n+1) \int^{1}_{0}\left[f_z(t)\right]^n dt
\end{equation}

where $f_z(t)=t(1-zt^2)$ is a polynomial in the complex variable $t$ and analytic in $z$.\\

\medskip
We shall denote the two branches of the square root of $z$ by $\pm\sqrt{z}$ where $\sqrt{z}$ is the branch with $\sqrt{1}=1$ and the square root $\sqrt{z}$ is holomorphic on the plane cut along the negative semi-axis.  Note that the function $f_z(t)$ has zeros at $t=0$ and $t=\pm\frac{1}{\sqrt{z}}$ while the critical points $f_z'(t)=0$ occur at $t=\pm\frac{1}{\sqrt{3z}}$.

%%%%%%%%%%%%%%%%%%%%%%%%%%%%
\section{Preliminary results}

In order to prove our main result, we will need the following lemmas.\\

First, we recall a version of the classical Enestr\"om-Kakeya theorem (cf. \cite{marden} p.136): If $0<a_0<a_1<\ldots <a_n$, then all zeros of the polynomial $p(z)=a_0+a_1z+\ldots +a_nz^n$ lie in the unit disk $|z|<1$.

\begin{lem}\label{enestromk}
The zeros of ${}_2F_1\left(-n,\frac{n+1}{2};\frac{n+3}{2};z\right)$ are contained in the disk $|z|<n+1$.
\end{lem}

\begin{proof}
Let $$F_n(z) = {}_2F_1 \left(-n, \frac{n+1}{2}; \frac{n+3}{2}; z \right) = c_0+c_1 z+\ldots +c_n z^n,$$ where
\begin{equation}\label{coeff}
c_m = \frac{(-n)_m \left(\frac{n+1}{2}\right)_m}{\left(\frac{n+3}{2}\right)_m m!}%= (-1)^m {n \choose m} \; \frac{n+1}{n+2m+1}
\end{equation}

A straightforward computation shows that $$\left|\frac{c_n}{c_{n-1}}\right|> \frac{1}{n+1}, \quad \textrm{for } n>1,$$ which implies that $$\frac{-(n+1)c_m}{c_{m-1}}>1, \qquad m=1,2,\ldots , n.$$ It follows immediately that the coefficients of the polynomial
\begin{align*}
p(z)=F_n\left(-(n+1)z\right)&=c_0- c_1(n+1)z+\ldots + (-1)^n(n+1)^nz^n\\
&= a_0+ a_1z +\ldots + a_n z^n
\end{align*}

are positive and increasing: $0<a_0<a_1<\ldots <a_n$.  By the Enestr\"om-Kakeya theorem, the zeros of $F_n\left(-(n+1)z\right)$ lie in the unit disk $|z|<1$ and therefore the zeros of $F_n\left(z\right)$ lie in the disk $|z|<n+1.$
\end{proof}

\bigskip
\begin{lem}\label{lastlem}
The polynomial ${}_2F_1\left(-n,\frac{n+1}{2};\frac{n+3}{2};z\right)$ has at least one zero outside the unit circle $|z|=1$.
\end{lem}

\begin{proof}
We have 
\begin{align*}
{}_2F_1\left(-n,\frac{n+1}{2}; \frac{n+3}{2};z\right)&= c_0+c_1 z+\ldots +c_n z^n\\
&=c_n(z-k_1)(z-k_2)\ldots (z-k_n)
\end{align*}

where $k_j$ is the $j^{\textrm{th}}$ zero of the polynomial, $j=1,2,\ldots ,n$.  Now $$c_0=c_nk_1k_2\ldots k_n(-1)^n$$ so that $$|k_1k_2\ldots k_n|=\left|\frac{c_0}{c_n}\right|.$$ Also, from (\ref{coeff}), we see that $$\left|\frac{c_0}{c_n}\right| = \frac{3n+1}{n+1} >1,$$ and so the product of the zeros has modulus greater than 1.  Therefore at least one zero of the polynomial ${}_2F_1\left(-n, \frac{n+1}{2}; \frac{n+3}{2}; z\right)$ must be outside the unit circle $|z|=1$.
\end{proof} 

\medskip
\begin{lem}\label{basins}
If Re$(\sqrt{z})>\frac{1}{\sqrt{3}}$, the function $|f_z(t)|$ given in (\ref{smiley}) has a unique path of steepest ascent from $\frac{1}{\sqrt{z}}$ to 1.  If $0< \textrm{Re} (\sqrt{z})<\frac{1}{\sqrt{3}}$, there is a  unique path of steepest ascent from 0 to 1.  
\end{lem}

\begin{figure}[!hbt]
\includegraphics[scale=0.65]{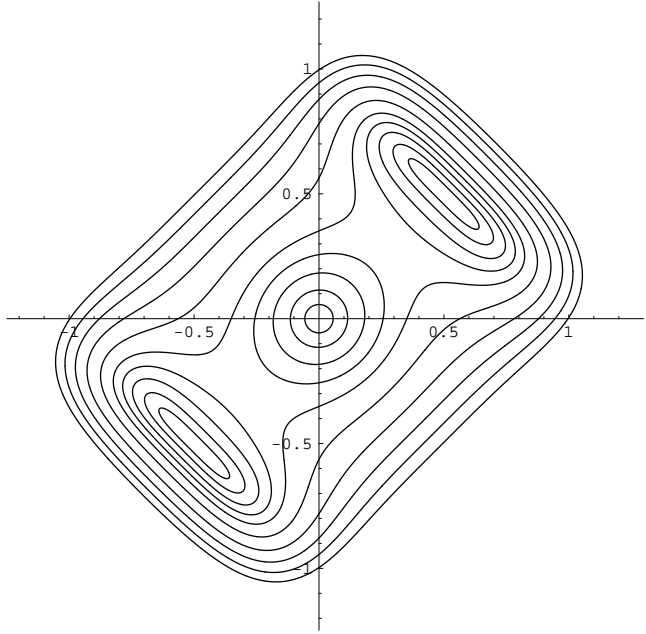}
\caption{The level curves of $f_z(t)$ \label{contour}}
\end{figure}

\begin{proof}
First we note that $|f_z(1)|>\left|f_z\left(\frac{1}{\sqrt{3z}}\right)\right| \Leftrightarrow |z(1-z)^2|>\frac{4}{27}$ and that the equivalence for the reverse inequality also holds.\\

The lines through the saddle-points $t=\pm \frac{1}{\sqrt{3z}}$ perpendicular to the linear segment from $-\frac{1}{\sqrt{z}}$ through 0 to $\frac{1}{\sqrt{z}}$ are ``continental divides'' that separate the $t$-plane into three basins containing $-\frac{1}{\sqrt{z}}$, 0 and $\frac{1}{\sqrt{z}}$ respectively.\\

Any point in the 0-basin is joined to 0 by a unique path of steepest decent, orthogonal to the level curves of $f_z(t)$.  Points in the $\frac{1}{\sqrt{z}}$-basin and $-\frac{1}{\sqrt{z}}$-basin can be similarly joined to $\frac{1}{\sqrt{z}}$ and $-\frac{1}{\sqrt{z}}$ respectively.\\

The point 1 will be located in either the 0-basin or the $\frac{1}{\sqrt{z}}$-basin, but not in the $-\frac{1}{\sqrt{z}}$-basin.  Multiplying each point in the $t$-plane by $\frac{\sqrt{z}}{|z|}$ rotates the figure so that all three zeros of $f_z(t)$ move to the real axis and the lines through the saddle-points $-\frac{1}{\sqrt{3z}}$ and $\frac{1}{\sqrt{3z}}$ are carried to the vertical lines through $-\frac{1}{\sqrt{3}|z|}$ and $\frac{1}{\sqrt{3}|z|}$ respectively.\\

The point 1 will then be in the $\frac{1}{\sqrt{z}}$-basin if and only if $\textrm{Re}\left(\frac{\sqrt{z}}{|z|}\right) > \frac{1}{\sqrt{3}|z|}$ which is equivalent to the condition $\textrm{Re}\left(\sqrt{z}\right) > \frac{1}{\sqrt{3}}.$ Similarly, the point 1 will then be in the 0-basin if and only if $\frac{-1}{\sqrt{3}}< \textrm{Re}\left(\sqrt{z}\right) < \frac{1}{\sqrt{3}}$ or, more precisely, $0< \textrm{Re}\left(\sqrt{z}\right) < \frac{1}{\sqrt{3}}$.
\end{proof}

%%%%%%%%%%%%%%%%%%%%%%%%%%%%%%%%%%%%%%%%%%%%%%%%%%%%%%%%%%%%%%%
\section{The region $\textrm{Re}(z)<\frac{1}{3}$}\label{zlthanthird}
\setcounter{equation}{0}

We consider the two possibilities illustrated in Theorem \ref{basins} separately (refer to Figure 1).  The first is the case where $\frac{-1}{\sqrt{3}}< \textrm{Re}\left(\sqrt{z}\right) < \frac{1}{\sqrt{3}}.$  This section of the $z$-plane is the area to the left of the parabola with vertex $\frac{1}{3}$  and intercepts $\frac{2}{3}i$ and $-\frac{2}{3}i$.  Thus all points $z$ satisfying $\frac{-1}{\sqrt{3}}< \textrm{Re}\left(\sqrt{z}\right) < \frac{1}{\sqrt{3}}$ will lie to the left of the vertical line Re$(z)=\frac{1}{3}$.\\

We shall prove that no zeros of ${}_2F_1\left(-n,\frac{n+1}{2};\frac{n+3}{2};z \right)$ are possible for Re$(z)<\frac{1}{3}$ and, therefore, no zeros of ${}_2F_1\left(-n,\frac{n+1}{2};\frac{n+3}{2};z \right)$ can occur for $\frac{-1}{\sqrt{3}}< \textrm{Re}\left(\sqrt{z}\right) < \frac{1}{\sqrt{3}}$.

\begin{thrm}\label{leftregion}
For sufficiently large $n$, the polynomial ${}_2F_1\left(-n,\frac{n+1}{2};\frac{n+3}{2};z \right)$ has no zeros in the region Re$(z)<\frac{1}{3}$.
\end{thrm}

\begin{proof} From (\ref{smiley}), we know that 
\begin{align*}
{}_2F_1\left(-n,\frac{n+1}{2};\frac{n+3}{2};z\right) &= 
(n+1)\int^{1}_{0} \left[f_z(t)\right]^n dt\\
&= (n+1) \int^{1}_{0} \left[t(1-zt^2)\right]^n dt.
\end{align*}

Lemma \ref{basins} ensures that for Re$(z)<\frac{1}{3}$, there is a unique path of steepest ascent from 0 to 1.  We may thus evaluate the integral involved in (\ref{smiley}) over this path.  In order to find the path of steepest ascent, we use that fact that $f_z(t)=t(1-zt^2)$ will have constant argument along this path (cf. \cite{copson}) so that we can parametrise the path by letting $$f_z(t)=f_z(1)r \qquad 0\leq r\leq 1,$$ or equivalently
\begin{equation}\label{implicit} 
t(1-zt^2)=r(1-z). 
\end{equation}

Then $$(1-3zt^2)dt=(1-z)dr$$ and our hypergeometric polynomial can be rewritten as $${}_2F_1\left(-n, \frac{n+1}{2}; \frac{n+3}{2}; z\right)=(n+1)(1-z)^{n+1} \int^{1}_{0} \frac{r^n}{1-3zt^2} dr$$%\label{lintf}

where $t=t(r)$ is defined implicitly by (\ref{implicit}), with $t(0)= 0$ and $t(1)=1$.\\

Any zeros $z=z_{nj}$ of ${}_2F_1\left(-n, \frac{n+1}{2}; \frac{n+3}{2}; z \right)$ in the region Re$(z)<\frac{1}{3}$ must satisfy $$(n+1)(1-z)^{n+1} \int^{1}_{0} \frac{r^n}{1-3zt^2} dr=0,$$ or equivalently, using (\ref{implicit})

\begin{equation}\label{negzeroint}
n \int^{1}_{0} \frac{(1-zt^2)tr^{n-1}}{1-3zt^2} dr=0.
\end{equation}

\bigskip
We will prove that the integral in (\ref{negzeroint}) is bounded away from 0 and hence deduce that no zeros of ${}_2F_1\left(-n,\frac{n+1}{2};\frac{n+3}{2};z\right)$ can lie in the half-plane Re$(z)<\frac{1}{3}$.\\

If the zeros are restricted by the inequality $\left| z- \frac{1}{3} \right| \geq \epsilon$ for some $\epsilon >0$, then the denominator of the integrand in (\ref{negzeroint}) satisfies $\left|1-3zt^2\right|\geq \delta >0$, where $\delta$ is independent of $z$.  Thus for any fixed $\rho$ with $0<\rho<1$, we have 

\begin{align*}
n\left|\int^{\rho}_{0} \frac{(1-zt^2)tr^{n-1}}{1-3zt^2} dr\right| & \leq n\int^{\rho}_{0} \frac{|(1-zt^2)t|}{|1-3zt^2|}\; r^{n-1} dr\\
& \leq Cn\int^{\rho}_{0} |(1-zt^2)t|\; r^{n-1} dr
\end{align*}

since $\frac{1}{\left|1-3zt^2\right|}\leq C$ for some constant $C$.\\

\bigskip
Lemma \ref{enestromk} states that the zeros of ${}_2F_1\left(-n,\frac{n+1}{2};\frac{n+3}{2};z\right)$ are contained in the disk $|z|<n+1$.  Thus 

\begin{align}
n\left|\int^{\rho}_{0} \frac{(1-zt^2)tr^{n-1}}{1-3zt^2} dr\right| & \leq Cn\int^{\rho}_{0} |(1-zt^2)t|\; r^{n-1} dr\nonumber\\
& \leq Cn^2\; \int^{\rho}_{0} r^{n-1} dr\nonumber\\
& \leq Cn\; \rho^{n} \rightarrow 0\label{0torho}
\end{align}

as $n\rightarrow \infty$ since $0<\rho <1$.  On the other hand, for $\rho$ sufficiently close to 1, the integral 
\begin{equation}\label{rhoto1}
n \int_{\rho}^{1} \frac{(1-zt^2)tr^{n-1}}{1-3zt^2} dr
\end{equation}

is bounded away from zero.  To prove this, we first note that since the path $t=t(r)$ must lie on the same side of the ``continental divide'' as the point 1 (cf. proof of Lemma \ref{basins}), we know that Re$(zt^2)<\frac{1}{3}$.  Our restriction on $z$ further ensures that $\left|zt^2-\frac{1}{3}\right|>\frac{\epsilon}{2}$ for $t$ sufficiently near 1.\\

\bigskip
Now the linear fractional mapping $$\omega= \phi (\zeta)= \frac{1-\zeta}{1-3\zeta}$$ sends the region $$\left\{\zeta: \textrm{Re} (\zeta)<\frac{1}{3}, \left|\zeta - \frac{1}{3} \right| >\frac{\epsilon}{2}\right\}$$ onto a semidisk to the right of the vertical line Re$(\omega)=\frac{1}{3}$.  It follows that 

\[\textrm{Re}\left\{\frac{(1-zt^2)t}{1-3zt^2}\right\}>\frac{1}{6}\]

when $t$ is close enough to 1.  This shows that

\[\textrm{Re}\left\{n\int_{\rho}^{1} \frac{(1-zt^2)tr^{n-1}}{1-3zt^2}\; dr\right\}>\frac{n}{6}\int_{\rho}^{1} r^{n-1}\; dr > \frac{1}{12}\]

for $\rho$ near 1 and all $z$ satisfying Re$(z)<\frac{1}{3}$. Combining this with (\ref{0torho}), we see that for sufficiently large $n$, the polynomial ${}_2F_1\left(-n,\frac{n+1}{2};\frac{n+3}{2};z\right)$ can have no zeros in the region Re$(z)<\frac{1}{3}$.
\end{proof}

\bigskip
Thus, if any zeros exist in the region Re$(z)\leq \frac{1}{3}$, they must converge uniformly to the point $\frac{1}{3}$ as $n\rightarrow\infty$ (since we had the additional restriction of $\left|z-\frac{1}{3}\right|>\epsilon$ for some $\epsilon>0$).  In common with the analysis in \cite{durenguillou}, we are unable to show that the polynomial never has zeros in this region, although numerical evidence generated by Mathematica suggests that this is the case.\\

%%%%%%%%%%%%%%%%%%%%%%%%%%%%%%%%%%%%%%%%%%%%%%
\section{The asymptotic zero distribution in the region Re$(z)>\frac{1}{3}$}
\setcounter{equation}{0}

We know from section \ref{zlthanthird} that for $n$ sufficiently large, the zeros of ${}_2F_1\left(-n,\frac{n+1}{2};\frac{n+3}{2};z\right)$ lie in the region Re$(z)>\frac{1}{3}$.\\

\medskip
From Lemma \ref{lastlem}, we know that at least one of the zeros of ${}_2F_1\left(-n,\frac{n+1}{2};\frac{n+3}{2};z\right)$ lies outside the unit circle $|z|=1$.  Since there are no zeros to the left of Re$(z)=\frac{1}{3}$, we know that this polynomial has at least one zero in the region Re$(z)>\frac{1}{3}$.\\

\medskip
We are now in a position to prove our theorem.

\begin{thrm}\nonumber
The zeros of the hypergeometric polynomial $${}_2F_1\left(-n, \frac{n+1}{2}; \frac{n+3}{2};z \right)$$ approach the section of the lemniscate $$\left\{z: \left|z(1-z)^2\right|=\frac{4}{27}; \textrm{Re}(z)>\frac{1}{3}\right\},$$ as $n\rightarrow \infty$.
\end{thrm}

\begin{proof}
From (\ref{smiley}), we know that $${}_2F_1\left(-n,\frac{n+1}{2};\frac{n+3}{2};z\right) = (n+1) \int^{1}_{0} \left[t(1-zt^2)\right]^n dt.$$ Lemma \ref{basins} guarantees that there is a unique path of steepest ascent from $\frac{1}{\sqrt{z}}$ to 1.  We may thus deform the path of integration in (\ref{smiley}) to write

\begin{equation}\nonumber
\int^{1}_{0} \left[f_z(t)\right]^n dt = \int^{\frac{1}{\sqrt{z}}}_{0} \left[f_z(t)\right]^n dt +\int^{1}_{\frac{1}{\sqrt{z}}} \left[f_z(t)\right]^n dt
\end{equation}

following the linear path from 0 to $\frac{1}{\sqrt{z}}$ and then the unique path of steepest ascent from $\frac{1}{\sqrt{z}}$ to 1.  The linear path from 0 to $\frac{1}{\sqrt{z}}$ is orthogonal to the level curves of $f_z(t)$ and is therefore the path of steepest ascent from 0 to the saddle-point $\frac{1}{\sqrt{3z}}$, followed by the path of steepest descent to $\frac{1}{\sqrt{z}}$.\\

\medskip
Any zero $z=z_{nj}$ of ${}_2F_1\left(-n, \frac{n+1}{2}; \frac{n+3}{2}; z \right)$ in the region Re$(z)>\frac{1}{3}$ must satisfy

\begin{equation}\label{gintf}
\int^{\frac{1}{\sqrt{z}}}_{0} \left[f_z(t)\right]^n dt +\int^{1}_{\frac{1}{\sqrt{z}}} \left[f_z(t)\right]^n dt=0
\end{equation}

\medskip
where the integrals are taken over paths of steepest ascent or descent.\\

Making the substitution $s=zt^2$ for $0\leq s\leq 1$, we obtain

\begin{align}
\int^{\frac{1}{\sqrt{z}}}_{0} \left[f_z(t)\right]^n dt &= \int^{\frac{1}{\sqrt{z}}}_{0} \left[t(1-zt^2) \right]^n dt\nonumber\\
&= \frac{1}{2(\sqrt{z})^{n+1}} \int^{1}_{0} s^{(n+1)/2} (1-s)^n ds\nonumber\\
&= \frac{1}{2(\sqrt{z})^{n+1}} \; \frac{\Gamma\left( \frac{n+1}{2}\right) \Gamma(n+1)}{\Gamma \left(\frac{3n+3}{2}\right)}.\nonumber%\label{gintf1}
\end{align}

\bigskip
Using Stirling's approximation $$\Gamma (n+1) = e^{-n} n^n \sqrt{2\pi n}\left(1+\frac{1}{12n}+\frac{1}{288n^2}+O\left(\frac{1}{n^3}\right)\right), \qquad n\rightarrow \infty,$$ we then have $$\int^{\frac{1}{\sqrt{z}}}_{0} \left[f_z(t)\right]^n dt = \frac{\sqrt{2\pi}}{3\sqrt{n}(\sqrt{z})^{n+1}} \left(\frac{2}{\sqrt{27}}\right)^n\left[1 + O\left(\frac{1}{n}\right)\right],$$%\label{intstirling}

\bigskip
as $n\rightarrow \infty$.  The second integral on the right-hand side of (\ref{gintf}) requires that we find the path of steepest ascent from $\frac{1}{\sqrt{z}}$ to 1.  We again use the fact that $f_z(t)$ will have constant argument along this path (cf. \cite{copson}) so that we can parametrise this path by letting $$f_z(t)=f_z(1)r \qquad 0\leq r\leq 1,$$ yielding $$\int^{1}_{\frac{1}{\sqrt{z}}} \left[f_z(t)\right]^n dt = (1-z)^{n} \int^{1}_{0} \frac{(1-zt^2)tr^{n-1}}{1-3zt^2} dr$$%\label{gintf2}

where $t=t(r)$ is defined implicitly by (\ref{implicit}), with $t(0)= \frac{1}{\sqrt{z}}$ and $t(1)=1$.  Therefore, any zero $z=z_{nj}$ of ${}_2F_1\left(-n, \frac{n+1}{2}; \frac{n+3}{2}; z \right)$ in the region Re$(z)>\frac{1}{3}$ must asymptotically satisfy $$\left(\frac{2}{\sqrt{27}}\right)^n \; \frac{\sqrt{2\pi}}{3\sqrt{n}(\sqrt{z})^{n+1}}\left\{ 1+O\left(\frac{1}{n}\right)\right\} +(1-z)^{n} \int_{0}^{1} \frac{(1-zt^2)tr^{n-1}}{1-3zt^2} dr = 0,$$ as $n\rightarrow \infty$, or equivalently

\begin{equation}\label{gzeros}
(\sqrt{z})^{n+1}(1-z)^{n} \int_{0}^{1} \frac{(1-zt^2)tr^{n-1}}{1-3zt^2} dr = -\left(\frac{2}{\sqrt{27}}\right)^n \; \frac{\sqrt{2\pi}}{3\sqrt{n}}\left\{1+O\left(\frac{1}{n}\right)\right\}.
\end{equation}

\medskip
as $n\rightarrow \infty$.  Taking moduli and $n^{\textrm{th}}$ roots on both sides of (\ref{gzeros}), we obtain $$|\sqrt{z}|^{\frac{1}{n}} \left|\sqrt{z}(1-z)\right| \left|\int_{0}^{1} \frac{(1-zt^2)tr^{n-1}}{1-3zt^2} dr\right|^{\frac{1}{n}} = \left|\frac{2}{\sqrt{27}}\right| \; \left(\frac{\sqrt{2\pi}}{3\sqrt{n}}\right)^{\frac{1}{n}} \left\{1+ O\left(\frac{1}{n}\right)\right\}^{\frac{1}{n}}.$$%\label{nthroots}

\medskip
It is straightforward to check that, as $n\rightarrow \infty$, $$\left|\int_{0}^{1} \frac{(1-zt^2)tr^{n-1}}{1-3zt^2} dr\right|^{1/n}$$ converges to 1 uniformly in $z$ and the zeros $z=z_{nj}$ of ${}_2F_1\left(-n,\frac{n+1}{2};\frac{n+3}{2};z\right)$ in the region Re$(z)>\frac{1}{3}$ approach the lemniscate $$\left|\sqrt{z}(1-z)\right|=\frac{2}{\sqrt{27}}$$ or equivalently \[\left|z(1-z)^2\right|=\frac{4}{27}.\qedhere \]
\end{proof}

\medskip
In addition, we note that by taking $n^{\textrm{th}}$ roots on both sides of (\ref{gzeros}), for large $n$, there are $n$ points satisfying (\ref{gzeros}), distinguished by the $n$ choices of $\sqrt[n]{-1}$.  All of these points are zeros of the polynomial, spread out near the right-hand branch of the lemniscate.\\

\newpage
Figure \ref{lemniscates} shows numerical plotting of the zeros for $n$ ranging from 5 to 60.

\begin{figure}[!hbt]
\includegraphics[scale=0.7]{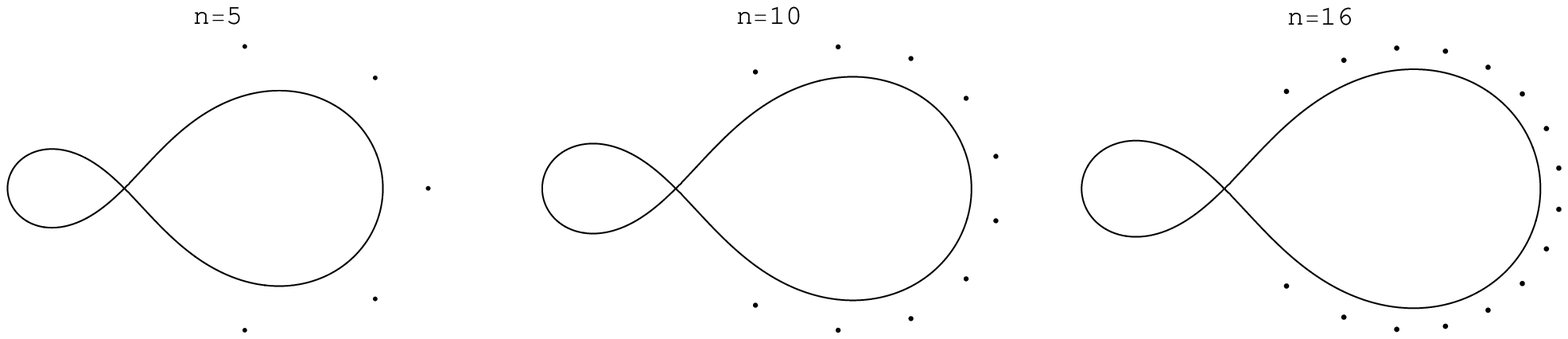}
\includegraphics[scale=0.65]{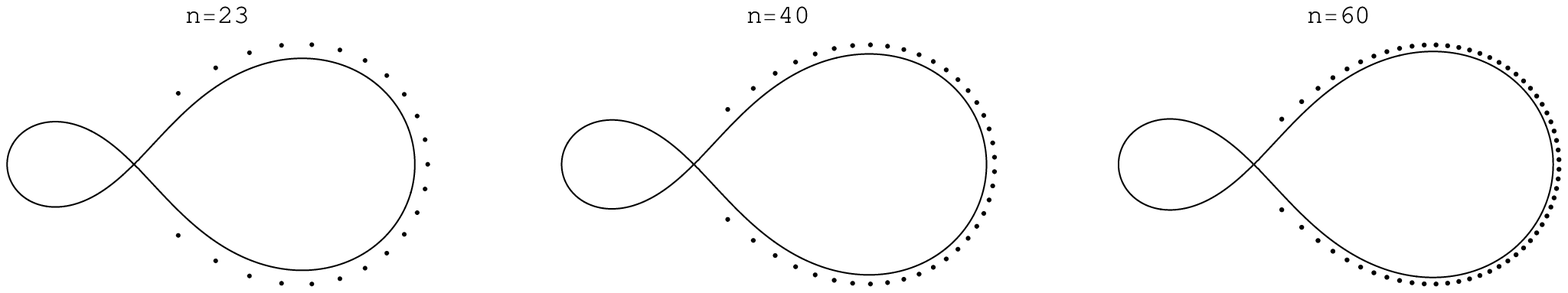}
\caption{The curve $\left|z(1-z)^2\right|=\frac{4}{27}$ and the zeros of ${}_2F_1\left( -n,\frac{n+1}{2};\frac{n+3}{2};z\right)$ for $n=5, 10, 16, 23, 40, 60.$\label{lemniscates}}
\end{figure}

\end{document}